\newcommand{\beq}{\begin{equation}}
\newcommand{\eeq}{\end{equation}}
\newcommand{\barr}{\begin{eqnarray}}
\newcommand{\earr}{\end{eqnarray}}
\documentclass{amsart}

\theoremstyle{plain}

\newtheorem{zvk}[subsection]{Zariski-van Kampen Theorem}

\newtheorem{prop}[subsection]{Proposition} 
\newtheorem{prop1}[subsubsection]{Proposition}

\theoremstyle{definition}
\newtheorem{rem}[subsection]{Remark}
\newtheorem{rem1}[subsubsection]{Remark}
 
\newtheorem{example}[subsection]{Example}

\usepackage{tikz} 
\usetikzlibrary{matrix,calc,arrows,decorations.markings}

\usepackage{amssymb}
\usepackage{graphicx}
\usepackage{hyperref}

\newcommand\enet[1]{\renewcommand\theenumi{#1}
\renewcommand\labelenumi{\theenumi}}

\begin{document}

\thanks{Both authors are partially supported by MTM2007-67908-C02-01.}
  
\title{On the topology of hypocycloids}

 \newcommand{\shortauthor}{{E. Artal and J.I. Cogolludo}}

\author[E. Artal]{Enrique Artal Bartolo}
\author[J.I. Cogolludo]{Jos\'e Ignacio Cogolludo-Agust\'{\i}n}
\address{Dpto. de Matem\'{a}ticas, Facultad de Ciencias, IUMA\\
  Universidad de Zaragoza, 50009 Zaragoza, Spain}
\email{artal@unizar.es,jicogo@unizar.es}


\begin{abstract}
Algebraic geometry has many connections with physics: string theory, enumerative geometry, and mirror symmetry, among others.
In particular, within the topological study of algebraic varieties physicists focus on aspects involving symmetry and non-commutativity. In this paper, we study a family of classical algebraic curves, the hypocycloids, which have links to
physics via the bifurcation theory. The topology of some of these curves plays an important role in string theory~\cite{cer:94} 
and also appears in Zariski's foundational work~\cite{zr:29}. We compute the fundamental groups of some of these curves and show that they 
are in fact Artin groups.
\end{abstract}
\maketitle

\section{Introduction}

Hypocycloid curves have been studied since the Renaissance (apparently D\"urer in 1525 described epitrochoids in general 
and then Roemer in 1674 and Bernoulli in 1691 focused on some particular hypocycloids, like the astroid, see~\cite{plane-curves}). 
Hypocycloids are described as the roulette traced by a point $P$ attached to a circumference $S$ of radius~$r$ rolling about 
the inside of a fixed circle $C$ of radius $R$, such that $0<\rho=\frac{r}{R}<\frac{1}{2}$ (see Figure~\ref{fig-hypo}). 
If the ratio $\rho$ is rational, an algebraic curve is obtained. The simplest (non-trivial) hypocycloid is called 
the deltoid or the \emph{Steiner curve} and has a history of its own both as a real and complex curve.

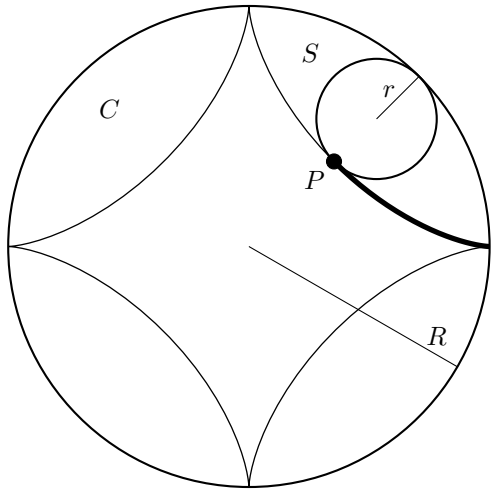
\begin{figure}[ht]
\centering
\begin{tikzpicture}[scale=.4,vertice/.style={draw,circle,fill,minimum size=0.2cm,inner sep=0}]
\draw[thick,black] (0,0) circle (8cm) node [above left=45pt] {$C$};
\draw[thick,black] (45:6) circle (2cm) node [above left=18pt] {$S$};
\node[vertice] at (45:4) {};
\node[below left] at (45:4) {$P$};
\draw (0,0)--(-30:8) node[pos=.9,above] {$R$};
\draw (45:6)--(45:8) node[pos=.3,above] {$r$};
\draw [line width=2pt] plot [parametric,domain=0:pi/4]
        ({8*cos(deg(\x))^3},{8*sin(deg(\x))^3});
\draw [line width=.5pt] plot [parametric,domain=pi/4:2*pi,samples=100]
        ({8*cos(deg(\x))^3},{8*sin(deg(\x))^3});
\end{tikzpicture}
\caption{Hypocycloid}\label{fig-hypo}
\end{figure}

Hypocycloids first appeared as trajectories of motions or integral solutions of vector fields, describing physical phenomena.
Modern physics also finds these objects useful. For instance, in the context of superstring compactifications of Calabi-Yau threefolds, certain Picard-Fuchs equations arise naturally. The monodromy group of such equations is the 
\emph{target duality group} acting on the moduli of string theory and can be computed as the fundamental group of the complement 
of the bifurcation locus in a deformation space. 

Using the celebrated Lefschetz-Zariski theorems of hyperplane sections (in the homotopy setting), these monodromy groups can 
be recovered in the context of complements of complex algebraic projective curves. 

For this reason, braid monodromies and fundamental groups of complements of plane curves have been intensively studied not only by
by mathematicians, but also by physicists in the past decades. 

Our purpose is to investigate the topology of the complement of some of those interesting hypocycloids using their 
symmetries and the structure of their affine and projective singularities in a very effective way. In order to do so, 
we need to introduce Zariski-van Kampen method~\cite{zr:29,vk:33}, braid monodromies and Chebyshev polynomials and exploit their properties.

\section{First properties of complex hypocycloids}\label{sec2}

Let us construct a parametrization of a hypocycloid as a real curve. Since $\rho$ is a positive rational number,
it admits the irreducible form $\rho:=\frac{\ell}{N}$, where $\ell$ and $N$ are coprime positive integers.
Also note that $\rho$ and $1-\rho$ define the same curve, hence $\rho\in(0,\frac{1}{2})$, i.e., $k:=N-\ell>\ell$,
will be assumed. For simplicity, the external circle $\mathcal{C}$ can be assumed to have radius $1$. 
If $C^{\mathbb R}_{k,\ell}$ denotes the real hypocycloid given by $k$ and $\ell$, it is not difficult to prove that
\beq\label{par1-hypo}
X_{k,\ell}(\theta)=\frac{\ell\cos k \theta+ k\cos \ell \theta}{N},\quad
Y_{k,\ell}(\theta)=\frac{\ell \sin k \theta- k\sin \ell \theta}{N}
\eeq
provides a parametrization of $C^{\mathbb R}_{k,\ell}$.
This parametrization is useful for drawing the hypocycloid but a rational one is preferred.
Given $n\in\mathbb{N}$ we denote by $T_n$, $U_n$, and $W_n$ the Chebyshev polynomials defined by
\beq\label{cheby}
\cos n \theta=T_n(\cos \theta),\quad
\sin(n+1)\theta=\sin \theta\ U_n(\cos \theta), \quad
\sin\!\left(\!n\!+\!\frac{1}{2}\!\right)\!\theta = \sin \frac{\theta}{2}\ W_n(\cos \theta).
\eeq
Let us recall that $T_n$, $U_n$, and $W_n$ have degree~$n$ and they have zeroes at:
\beq\label{zeroes-cheby}
\begin{array}{lll}
T_n(x)=0 & \Leftrightarrow & x=\cos \left( \frac{(2r-1)\, \pi}{2n}\right)\quad r=1,...,n,\\
U_n(x)=0 & \Leftrightarrow & x=\cos \left( \frac{r\, \pi}{n+1}\right)\quad r=1,...,n,\\
W_n(x)=0 & \Leftrightarrow & x=\cos \left( \frac{r\, \pi}{n+\frac{1}{2}}\right)\quad r=1,...,n.
\end{array}
\eeq

The following rational parametrization is obtained:
\beq\label{par2-hypo}
x_{k,\ell}(t):=\frac{P_{k,\ell}\left(\frac{1-t^2}{1+t^2}\right)}{N}, \quad y_{k,\ell}(t):=\frac{2t\ Q_{k,\ell}\left(\frac{1-t^2}{1+t^2}\right)}{N(1+t^2)},
\eeq
where
$P_{k,\ell}(x):=\ell T_k \left(x\right)+ k T_\ell \left(x\right)$, and
$Q_{k,\ell}(t):=\ell U_{k-1} \left(x\right)- k U_{\ell-1} \left(x\right)$.

If the parameter $t$ is allowed to run along the complex numbers outside $\{\pm \sqrt{-1}\}$
one obtains a complex plane curve, which will be called the complex hypocycloid, or simply hypocycloid for short, 
and denoted by $C_{k,\ell}\subset \mathbb C^2$. Note that $C^{\mathbb R}_{k,\ell}\subset C_{k,\ell} \cap \mathbb R^2$.

Moreover, let us recall that any affine complex curve in $\mathbb C^2$ defined by rational parametric equations 
$t\mapsto \left(\frac{p_1(t)}{p_3(t)},\frac{p_2(t)}{p_3(t)}\right)$
can be embedded in $\mathbb{P}^2:=\mathbb C \mathbb P^2$, the complex projective plane, by homogenizing its parametric equations and 
removing denominators. This way, the parameter space becomes $\mathbb P^1$ and the new projective parametric 
equations become $[t:s]\mapsto [s^{d-d_1}p_1\left(\frac{t}{s}\right):s^{d-d_2}p_2\left(\frac{t}{s}\right):s^{d-d_3}p_3\left(\frac{t}{s}\right))$, where $d_i:=\deg p_i(t)$, and 
$d:=\max \{d_1,d_2,d_3\}$. The complex projective hypocycloid will be denoted by $\bar C_{k,\ell}\subset \mathbb P^2$.
Note that the former parameters $\{\pm \sqrt{-1}\}$ can be interpreted as the \emph{points at infinity} of the 
complex hypocycloid.

\begin{prop}
\label{prop-hypo}
The complex projective hypocycloid $\bar C_{k,\ell}$ is a rational curve of degree $2 k$ with the following properties:
\begin{enumerate}
\enet{\rm(\roman{enumi})}
\item \label{part-symmetry}
The curve $C_{k,\ell}$ is invariant by an action of the dihedral group $\mathbb{D}_{2 N}$.
\item \label{part-aff-sing}
The 
singular points 
of $C_{k,\ell}$
are only ordinary nodes and ordinary cusps arranged as follows: 
$N$~cusps, $N(\ell-1)$ (real) nodes, and $N(k-\ell-1)$ (non-real) nodes.
\item \label{part-inf-sing}
The intersection with the line at infinity consists of two points with local equations $u^k-v^{k-\ell}=0$ tangent 
to the line at infinity (these points are singular if $k-\ell>1$). 
\item \label{part-param}
The parametrization given in~\eqref{par2-hypo} is an immersion (outside the $N$ cusps and eventually the points at infinity
if $k-\ell>1$).
\end{enumerate}
\end{prop}

\begin{proof}
First, let us show that 
the rotation of angle $\frac{2\pi}{N}$ and the reflection
with respect to the horizontal axis given by the equation $\{y=0\}$ globally fix $C_{k,\ell}$. These two symmetries generate a dihedral group
of order $2N$, denoted by $\mathbb D_{2N}$.

The reflection is an immediate consequence of the fact that the first coordinate of the affine parametrization is
even $x_{k,\ell}(-t)=x_{k,\ell}(t)$, while the second coordinate is odd $y_{k,\ell}(-t)=-y_{k,\ell}(t)$, see~\eqref{par2-hypo}.

In order to visualize the rotation, it is more convenient to use trigonometric notation. One can check that
\[
\left(
\begin{matrix}
X_{k,\ell}\left(\theta+\frac{2\pi}{N}\right)\\
\\
Y_{k,\ell}\left(\theta+\frac{2\pi}{N}\right)
\end{matrix}
\right)
=
\left(
\begin{matrix}
\cos \frac{2\pi k}{N} & -\sin \frac{2\pi k}{N}\\
&\\
\sin \frac{2\pi k}{N} & \cos \frac{2\pi k}{N}
\end{matrix}
\right)
\left(
\begin{matrix}
X_{k,\ell}(\theta)\\
\\
Y_{k,\ell}(\theta)
\end{matrix}
\right),
\]
Since this is a rotation of degree $\frac{2\pi k}{N}$ and $\gcd(k,N)=1$ then
part~\ref{part-symmetry} follows.

Note that
\beq
\label{eq-param}
\varphi([t:s]:=\left[p_{k,\ell}(t,s): 2ts q_{k,\ell}(t,s): N (s^2+t^2)^k \right]
\eeq
is a parametrization of the projective hypocycloid $\bar C_{k,\ell}$, where
\[
p_{k,\ell}(t,s):=(s^2+t^2)^k P_{k,\ell}\left(\frac{s^2-t^2}{s^2+t^2}\right),
\quad
q_{k,\ell}(t,s):=(s^2+t^2)^{k-1} Q_{k,\ell}\left(\frac{s^2-t^2}{s^2+t^2}\right).
\]
Note that $p_{k,\ell}$ is homogeneous of degree~$k$ and $q_{k,\ell}$ is homogeneous of degree~$k-1$.
Since the leading coefficient of $P_{k,\ell}$ is $2^{k-1}\ell$ and the one of $Q_{k,\ell}$ is $2^{k-2}\ell$,
see \cite{chebyshev},
then $\varphi([1:\pm \sqrt{-1}])=[1:\pm \sqrt{-1}:0]$ and hence this parametrization induces 
a well-defined surjective map from $\mathbb P^1$ to $\bar C_{k,\ell}$. Hence, outside a finite number of points
(those where the parametrization is not injective) $\bar C_{k,\ell}$ is isomorphic to $\mathbb P^1$, which 
implies that $\bar C_{k,\ell}$ is a rational curve. Also, its degree corresponds with the degree of any of its parametric
equations, namely $2k$.

It is a straightforward computation that
\[
\frac{d \varphi_1}{d t}=
-\frac{4k\ell ts^2}{N(s^2+t^2)^2}\left( U_{k-1}+U_{\ell-1}\right), \quad
\text{ and }\quad\quad
\frac{d \varphi_2}{d t}=
\frac{2k\ell s}{N(s^2+t^2)}\left( T_{k}-T_{\ell}\right).
\]
One should consider two different cases:
\begin{itemize}
\item If $N$ is even, then one can use the following two formulas (see~\cite{chebyshev}):
$U_{k-1}+U_{\ell-1}=2T_{\frac{k-\ell}{2}}U_{\frac{k+\ell}{2}-1}$
and $(T_{k}(x)-T_{\ell}(x))+x(U_{k-1}(x)+U_{\ell-1}(x))=(U_k(x)+U_{\ell-2}(x))$. Therefore:
\beq
\label{eq-factor-even}
\begin{array}{l}
\dfrac{d \varphi_1}{d t}=
-\dfrac{8k\ell t s^2}{N(s^2+t^2)^2} T_{\frac{k-\ell}{2}}U_{\frac{k+\ell}{2}-1},\\
\\
2ts\dfrac{d \varphi_2}{d t}-(s^2-t^2)\dfrac{d \varphi_1}{d t}=
\dfrac{8k\ell ts}{N(s^2+t^2)}T_{\frac{k-\ell}{2}+1}U_{\frac{k+\ell}{2}-1}.
\end{array}
\eeq
Thus the common zeroes to $\varphi_1':=\frac{d \varphi_1}{d t}$ and $\varphi_2':=\frac{d \varphi_2}{d t}$ are given by
$$
t s U_{\frac{k+\ell}{2}-1}\left(\frac{s^2-t^2}{s^2+t^2}\right)=0,
$$ 
that is,
\[
\left\{ (t,s) \mid t=0, s=0, \text{ or } \frac{s^2-t^2}{s^2+t^2}=\cos \left( \frac{2r \pi}{k+\ell} \right), r=1,...,\frac{k+\ell}{2}-1\right\}.
\]
This makes a total of $2\left( \frac{k+\ell}{2}-1\right) + 2=N$ singularities. Using the previous equations one can check that 
the order of $t$ in $(\varphi_1',\varphi_2')$ is (1,2) and consequently such a singularity is an 
ordinary cusp of equation $y^2-x^3$ whose tangent is the line $L_0:=\{y=0\}$. By applying the symmetry~\ref{part-symmetry}, 
the remaining singularities are also cusps and their tangents are a rotation of $L_0$.

By the parity of $N$, the line $L_0$ intersects $\bar C_{k,\ell}$ at two cusps $\varphi([1:0])$ and $\varphi([0:1])$ with 
multiplicity 3 each, hence there are $2(k-3)$ extra points of intersection (counted with multiplicity). Let us denote by $\varphi([t_0:s_0])=P$ one such point. Due to the symmetry of $\bar C_{k,\ell}$ with respect to $L_0$, unless the tangent direction at $P$ is vertical, 
the curve $\bar C_{k,\ell}$ possesses a node at $P$. In order to prove that the $2(k-3)$ extra points of
intersection do in fact correspond to $(k-3)$ nodes, one just needs to check that $\varphi_2$ and $\varphi_1'$ do not have 
any common zeroes. In order to do so, it is enough to note that, according to~(\ref{eq-factor-even}), all the roots 
of $\varphi_1'$ that are not critical points of the parametrization are of the form 
$\frac{s^2-t^2}{s^2+t^2}=\cos \left( 2\frac{2r-1}{k-\ell}\pi\right)$, which are not roots of $\varphi_2$.
Applying the symmetry of order $N$ and the fact that the orbit of $L_0$ by such symmetry has $\frac{N}{2}$ lines,
one obtains the existence of $\frac{(k-3)N}{2}$ nodes. Let us denote this group of nodes by $A_1$.

Another group of nodes is placed on the line $L_{2N}$ which is the rotation of $L_0$ by an angle of $\frac{\pi}{N}$
radians. In order to do so, let us reparametrize the hypocycloid so that $L_{2N}$ is sent to the horizontal line.
This corresponds to switching $\ell \theta$ by $\ell \theta + \pi$ in the inner circle. The real equations are 
transformed as follows (compare with~(\ref{par1-hypo})):

\beq\label{eq-hypo-rotate}
\tilde{X}_{k,\ell}(\theta)=\frac{1}{N}\left(\ell\cos k \theta- k\cos \ell \theta\right),\quad
\tilde{Y}_{k,\ell}(\theta)=\frac{1}{N}\left(\ell \sin k \theta+ k\sin \ell \theta\right)
\eeq
which provide the rational parametrization:
\beq\label{eq-hypo-rotate2}
\tilde x_{k,\ell}(t):=\frac{\tilde P_{k,\ell}\left(\frac{1-t^2}{1+t^2}\right)}{N}, \quad \tilde y_{k,\ell}(t):=\frac{2t \tilde Q_{k,\ell}\left(\frac{1-t^2}{1+t^2}\right)}{N(1+t^2)},
\eeq
where
$\tilde P_{k,\ell}(x):=\ell T_k \left(x\right)- k T_\ell \left(x\right)$, and
$\tilde Q_{k,\ell}(x):=\ell U_{k-1} \left(x\right)+ k U_{\ell-1} \left(x\right)$.
Note that again $\tilde x_{k,\ell}(-t)=\tilde x_{k,\ell}(t)$ and $\tilde y_{k,\ell}(-t)=-\tilde y_{k,\ell}(t)$, 
and that the rotation of $\frac{2\pi}{N}$ radians is a symmetry of the curve. Using the formula
$U_{k-1}-U_{\ell-1}=2T_{\frac{k+\ell}{2}}U_{\frac{k-\ell}{2}-1}$ one can check that 
\[
\tilde \varphi_1':=\frac{\tilde \varphi_1}{d t}=-\frac{8k\ell st^2}{N(s^2+t^2)^2}T_{\frac{k+\ell}{2}}U_{\frac{k-\ell}{2}-1},
\]
and hence $\tilde \varphi_1'$ and $\tilde \varphi_2$ have common zeroes only at $t=0$ and $s=0$, which are
two vertical tangents. This shows that there are $(k-1)$ remaining nodes on $L_{2N}$ and, after applying the rotation, 
one finds $\frac{(k-1)N}{2}$ new nodes. Let us denote this group of nodes by $A_2$ and define $A=A_1\cup A_2$.

\item If $N$ is odd, then proceeding as above, one can use the formula:
$T_{k}(x)-T_{\ell}(x)=(1-x)W_{\frac{k+\ell-1}{2}}(x)W_{\frac{k-\ell-1}{2}}(x)$
and check that the critical points of the parametrization are given by $t=0$ and by 
$W_{\frac{k+\ell-1}{2}}\left(\frac{s^2-t^2}{s^2+t^2}\right)=0$, that is, 
$\{ (t,s) \mid t=0, \text{ or } \frac{s^2-t^2}{s^2+t^2}=\cos \left( \frac{2r \pi}{k+\ell} \right), r=1,...,\frac{k+\ell-1}{2}\}$.

Analogously as in the previous case, the singularity at $t=0$ is an ordinary cusp $y^2-x^3$ whose tangent is $L_0$.
There are $2k-3$ remaining points of intersection which are necessarily a vertical tangent and $k-2$ nodes, since 
$\varphi_2$ and $\frac{\varphi_1}{d t}$ do not have any common zeroes similarly as above. Again, after applying the 
rotation of order $N$ one can find $(k-2)N$ nodes. Let us denote this group of nodes by $A$.
\end{itemize}

Summarizing, we have found $\# A=(k-2)N$ nodes and $N$ ordinary cusps. Assuming~\ref{part-inf-sing} $\bar C_{k,\ell}$
contains also 2 singular points of type $u^k-v^{k-\ell}$. Using the genus formula:
\[
\frac{(2k-1)(2k-2)}{2}-(k-2)N-N-2\frac{(k-\ell-1)(k-1)}{2}-\alpha=0
\]
where the first summand is the virtual genus of a curve of degree $2k$, the second summand comes from the nodes,
the third summand comes from the cusps, the fourth one comes from the singularities at infinity, and the last one
comes from any further singularities. Since the equation has to equal zero due to the fact that $\bar C_{k,\ell}$
is rational, this forces $\alpha$ to be equal to zero, and thus $\bar C_{k,\ell}$ contains no further singularities.

To finish the proof of part~\ref{part-aff-sing} one needs to make sure that only $(\ell-1)N$ nodes are real.
If $N$ is odd this can be done by verifying that there are only $(\ell-1)$ real nodes on $L_0$. Note that the real 
nodes come from branches joining the cusps. The line $L_0$ contains one cusp and divides the set of remaining cusps 
into two groups of $\frac{N-1}{2}$ each. Since the real branches join a cusp with the $\ell$-th consecutive cusp,
there is a total of $2(\ell-1)$ branches crossing $L_0$ which lead to $(\ell-1)$ nodes. Again, applying the 
rotation one obtains the $(\ell-1)N$ real nodes. The case when $N$ is even is analogous. Part~\ref{part-aff-sing} 
will be proved if~\ref{part-inf-sing} is checked.

Let $P_1,P_2$ be the two points in $\bar{C}_{k,\ell}\cap L_\infty$, where $L_\infty$ is the line at infinity. 
Let us note that any reflection in $\mathbb{D}_{2 N}$ fixes $L_\infty$ and $\mathcal{C}$ and interchanges $P_1$ and $P_2$.
By Bezout's Theorem, $L_\infty\cdot\bar{C}_{k,\ell}=2 k$; because of the symmetry, $(L_\infty\cdot\bar{C}_{k,\ell})_{P_i}=k$.
On the other hand it is easily seen that the only intersection points of $\bar{\mathcal{C}}$ and $\bar{C}_{k,\ell}$
are the $N$ cusps and $P_1,P_2$ (compute $X(\theta)^2+Y(\theta)^2$). Since $\bar{\mathcal{C}}$ and $\bar{C}_{k,\ell}$ are transversal at the cusps, their intersection number at these points is $2$. One obtains the following:
\[
4 k=\bar{\mathcal{C}}\cdot\bar{C}_{k,\ell}=2(k+\ell)+2(\bar{\mathcal{C}}\cdot\bar{C}_{k,\ell})_{P_1}
\]
and $(\bar{\mathcal{C}}\cdot\bar{C}_{k,\ell})_{P_1}=(\bar{\mathcal{C}}\cdot\bar{C}_{k,\ell})_{P_2}=k-\ell$.

It is a standard fact in singularity theory that a locally irreducible curve germ intersecting
two smooth transversal branches with coprime multiplicities $p$ and $q$ has the same topological
type as $u^p-v^q=0$. Therefore \ref{part-inf-sing} follows. Note that if $k-\ell=1$ then the points at infinity
are smooth, otherwise they are singular: this also proves~\ref{part-param}.
\end{proof}

\section{Fundamental group of the complement of a curve}\label{sec-zvk}

Let us consider an affine complex plane algebraic curve $C\subset\mathbb{C}^2$. Let us denote
by $f(x,y)\in\mathbb{C}[x,y]$ a reduced equation of the curve $C$. For simplicity $f$ will be assumed to be 
monic as a polynomial in $y$ (this can be achieved by applying a generic change of coordinates
and dividing by a non-zero constant). Let $d$ be the degree of $f$ in $y$ (which may be smaller than
the total degree of $f$).

The classical Zariski-van Kampen method works as follows. One considers a generic vertical line $L$ in $\mathbb{C}^2$; 
the group $\pi_1(L\setminus C)$ is isomorphic to the free group $\mathbb{F}_d$, since $L\cap C$ consists of $d$ points.
A basis of loops in this group also generates $\pi_1(\mathbb{C}^2\setminus C)$ and the relations are obtained via the 
monodromy of this fibration, which is basically \emph{moving $L$ around} the non-generic vertical lines. Let us state 
it more precisely (see~\cite{acc:01a} for a more complete version).

For $t\in\mathbb{C}$, $L_t$ denotes the vertical line $x=t$. Let $NT:=\{t\in\mathbb{C}\mid\ L_t\not\pitchfork C\}$;
$NT$ is a finite set and it is the zero locus of the discriminant of $f$ with respect to $y$ (which is a polynomial in $x$).
If $t\notin NT$, then $C\cap L_t$ consists of $d$ points, and by the continuity of roots one can see $f$ as
a holomorphic mapping $\tilde{f}:\mathbb{C}\setminus NT\to V\setminus\Delta$, where:
\begin{itemize}
\item $V$ is the space of monic complex polynomials in one variable and degree $d$ (naturally isomorphic to $\mathbb{C}^d$
via the coefficients);
\item $\Delta$ is the discriminant of $V$, i.e., the set of polynomials in $V$ with multiple roots 
(which is a hypersurface of $V$).
\end{itemize}

The space $V\setminus\Delta$ can be naturally identified with the configuration space of $d$ different points in 
$\mathbb{C}$, whose fundamental group is the braid group $\mathbb{B}_d$ in $d$ strings. Let us recall the Artin 
presentation of this group:
\beq\label{artinpres}
\mathbb{B}_d=\left\langle \sigma_1,\dots,\sigma_{d-1}\Big\vert
\sigma_i\sigma_{i+1}\sigma_i=\sigma_{i+1}\sigma_i\sigma_{i+1},\ \scriptstyle{i=1,\dots,d-2},\ 
\displaystyle{[\sigma_i,\sigma_j]=1,}\ \scriptstyle{1\leq i<j-1<d-2}
\right\rangle.
\eeq
There is a free action of this group on the free group $\mathbb{F}_d$ with generators $a_1,\dots,a_d$ defined as follows:
\beq\label{braidaction}
a_i^{\sigma_j}:=
\begin{cases}
a_{i+1}&\text{ if } i=j\\
a_{i+1} a_i a_{i+1}^{-1}&\text{ if } i=j+1\\
a_i&\text{ if } i\neq j,j+1.
\end{cases}
\eeq
Note that for any $\tau\in\mathbb{B}_d$, $(a_d\cdot\ldots\cdot a_1)^\tau=a_d\cdot\ldots\cdot a_1$.

One can define the braid monodromy of $C$ with respect to the coordinates $x,y$ as follows. Let
$t_0\in\mathbb{C}\setminus NT$ and let $\mathbb{F}:=\pi_1(\mathbb{C}\setminus NT;t_0)$, which is a free group.
Then one defines $\nabla:\mathbb{F}\to\mathbb{B}_d$ as the morphism defined by $\tilde{f}$ at the level of fundamental 
groups (with a suitable identification of $\mathbb{B}_d$ with $\pi_1(V\setminus\Delta;\tilde{f}(t_0)$).

\begin{zvk}
The fundamental group of $\mathbb{C}^2\setminus C$ is the quotient of $\mathbb{F}_d$ by the subgroup
normally generated by $w^{-1} w^{\nabla(\tau)}$, $w\in\mathbb{F}_d$, $\tau\in\mathbb{B}_d$. If $b_1,\dots,b_r$ is a free 
generating system of $\mathbb{F}$, then it admits the following presentation: 
\beq\label{zvkpres}
\left\langle a_1,\dots,a_d\Big\vert\
a_i=a_i^{\nabla(b_j)},\ i=1,\dots,d-1,\ j=1,\dots,r
\right\rangle.
\eeq
\end{zvk}

\begin{rem}
A natural interpretation of $\pi_1(V\setminus\Delta;\tilde{f}(t_0))$ can be given when 
$f\in\mathbb{R}[x,y]$, $t_0\in\mathbb{R}$ and all the roots of $f(t_0,y)$ are real. 
Let $y_1>\dots>y_d$ be the roots of $f(t_0,y)$. Then $\sigma_i$
is the homotopy class of the mapping $[0,1]\to V\setminus D$,
\[
t\mapsto\{y_1,\dots,y_{i-1},c_i+r_i e^{\pi t\sqrt{-1}},c_i-r_i e^{\pi t\sqrt{-1}},y_{i+2},y_d\},
\]
where $c_i$ is the middle point between $y_i$ and $y_{i+1}$ and $r_i$ is half their distance. A similar argument
works without the \emph{real} assumptions.
\end{rem}

\begin{rem}
Let us assume again that $f\in\mathbb{R}[x,y]$, $t_0\in\mathbb{R}$ and all the roots of $f(t_0,y)$ are real. 
For the free group $\pi_1(L_{t_0}\setminus C;y_0)$, $y_0\gg y_1$, a basis can be chosen as follows. 
Fix a small radius $\varepsilon>0$.

Consider a lasso $a_i:=u_i\cdot\delta_i\cdot u_i^{-1}$ based at $y_0$ as follows. The path $u_i$ runs along the real line 
from $y_0$ to $y_i-\varepsilon$ and avoids the points $y_1,\dots,y_{i-1}$ by the upper semicircles of radius $\varepsilon$ 
centered at these points; the lasso $\delta_i$ runs counterclockwise along the circle of radius $\varepsilon$
centered at $y_i$.

\begin{figure}[ht]
  \centering
\begin{tikzpicture}[scale=2,vertice/.style={draw,circle,fill,minimum size=0.4cm,inner sep=0}]
\tikzset{flecha/.style={decoration={
  markings,
  mark=at position #1 with  {\arrow[scale=2]{>}}},postaction={decorate}}}

\coordinate (O) at (7,0);
\coordinate (A1) at (6,0);
\coordinate (A2) at (5,0);
\coordinate (Ap) at (4,0);
\coordinate (A3) at (3,0);
\coordinate (A4) at (2,0);

\draw[flecha=.55] (O)--($(A1)+.3*(1,0)$);
\draw[flecha=.5]  ($(A1)+.3*(1,0)$) arc (0:180:.3);
\draw[flecha=.55] ($(A1)-.3*(1,0)$)--($(A2)+.3*(1,0)$);
\draw[flecha=.75]  ($(A2)$) circle (0.3);

\node[vertice,fill=white] at (O) {};

\node[vertice] at (A1) {};
\node[vertice] at (A2) {};
\node[] at (Ap) {$\dots$};
\node[vertice] at (A3) {};
\node[vertice] at (A4) {};

\end{tikzpicture}
  \caption{An element of a standard geometric basis}\label{fig-base1}
\end{figure}
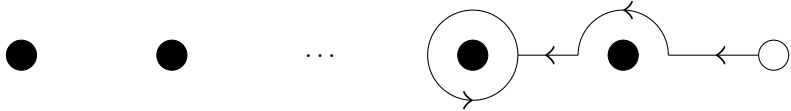

The \emph{ordered} basis $(a_1,\dots,a_d)$ is called a \emph{standard geometric basis}. Note that $a_i$ is a meridian
of the point $y_i$ (see~\cite{acc:01a} for a definition) and $(a_d\cdot\ldots\cdot a_1)^{-1}$ is a meridian of the point at infinity. These identifications give the geometric counterpart of the action \eqref{braidaction}. In the non-real case,
\emph{standard geometric bases} play the same role (see~\cite{acc:01a}).

\begin{figure}[ht]
  \centering
\begin{tikzpicture}[scale=.35]
\sf
\tikzset{
    partial ellipse/.style args={#1:#2:#3}{
        insert path={+ (#1:#3) arc (#1:#2:#3)}
    }
}

\coordinate (P0) at (15,-1);
\coordinate (Q1) at (19,6);
\coordinate (Q0) at (15,6);
\coordinate (P1) at (19,-1);
\coordinate (P) at (20.5,-1);
\coordinate (Q) at (20.5,6);

\definecolor{fillColor}{rgb}{0.52941,0.80784,1}
\path[draw=black,ultra thick,fill=fillColor] (22,7) -- (19,5) -- (12,5) -- (15,7)node[left=.5] {$a_1^{\sigma_1}=a_2$};
\path[draw=black,ultra thick,fill=cyan] (22,0) -- (19,-2) -- (12,-2) -- (15,0)  node[left=.5] {$a_1$};
\draw (P0)--($.8*(Q1)+.2*(P0)$) ;
\draw ($.9*(Q1)+.1*(P0)$)--(Q1) ;
\draw (P)--($(P1)+(.5,0)$);
\draw (P1) [partial ellipse=0:90:.5 and .3];
\draw (P1) [partial ellipse=130:360:.5 and .3];
\draw (P) -- (Q);
\draw[smooth] (Q) to [bend right=20]  ($(Q0)+(.5,0)$);
\draw (Q0) ellipse (.5 and .3);

\draw ($(P1)$) -- ($.4*(Q0)+.6*(P1)$) ;
\draw ($.6*(Q0)+.4*(P1)$) -- ($.8*(Q0)+.2*(P1)$) ;
\draw ($.9*(Q0)+.1*(P1)$)-- ($.94*(Q0)+.06*(P1)$) ;
\draw ($.98*(Q0)+.02*(P1)$)-- ($1*(Q0)+.00*(P1)$) ;
\begin{scope}[xshift=12cm]
\coordinate (P0) at (15,-1);
\coordinate (Q1) at (19,6);
\coordinate (Q0) at (15,6);
\coordinate (P1) at (19,-1);
\coordinate (P) at (20.5,-1);
\coordinate (Q) at (20.5,6);

\path[draw=black,ultra thick,fill=fillColor] (22,7) -- (19,5) -- (12,5) -- (15,7) node [right=2.5] {$a_2^{\sigma_1}=a_2 a_1 a_2^{-1}$};
\path[draw=black,ultra thick,fill=cyan] (22,0) -- (19,-2) -- (12,-2) -- (15,0)
node[right=2.5] {$a_2$};

\draw (P0)--($.8*(Q1)+.2*(P0)$) ;
\draw ($.9*(Q1)+.1*(P0)$)-- ($.94*(Q1)+.06*(P0)$) ;
\draw ($.97*(Q1)+.03*(P0)$)--(Q1) ;

\draw ($(P1)$) -- ($.4*(Q0)+.6*(P1)$) ;
\draw ($.6*(Q0)+.4*(P1)$) -- ($.8*(Q0)+.2*(P1)$) ;
\draw ($.92*(Q0)+.08*(P1)$)--(Q0) ;

\draw[smooth] (P) to ($(P1)+(.2,.2)$)  ; 
\draw ($(P1)+(-.5,.3)$) -- ($(P0)+(.5,0)$) ;
\draw (P0) [partial ellipse=0:50:.5 and .3];
\draw (P0) [partial ellipse=90:360:.5 and .3];

\draw (P) -- (Q);
\draw (Q) to [out=165,in=45]($(Q0)+(-1,0)$) to[out=-135,in=210] ($(Q1)-(.5,0)$)  ; 
\draw (Q1) ellipse (.5 and .3);
\end{scope}

\end{tikzpicture}%

  \caption{Geometric action of the braid group}\label{fig-genaccion}
\end{figure}
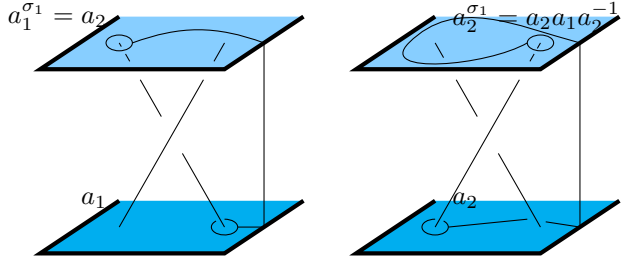

In general \emph{standard pseudogeometric bases} are preferred for the group $\mathbb{F}$ (if $NT\subset\mathbb{R}$); 
the only difference with geometric bases being that the condition on the position of the base points is weakened.
\end{rem}

\begin{example}\label{ex-rels}
Let us assume that $b_1$ corresponds to a small loop surrounding a point $t$ such that $L_t$ is an ordinary tangent line.
For suitable choices of loops and paths, $\nabla(b_1)=\sigma_1$ and the only non-trivial relation is given by $a_1=a_2$.
Analogously, for other non-transversal vertical lines $L_t$, one obtains the following braids and relations:
\begin{itemize}
\item If $L_t$ passes transversally through a node, then $\nabla(b_1)=\sigma_1^2$ and the only non-trivial relation 
is given by $[a_1,a_2]=1$.
\item If $L_t$ passes transversally through an ordinary cusp, then $\nabla(b_1)=\sigma_1^3$ and the only non-trivial 
relation is given by $a_1 a_2 a_1=a_2 a_1 a_2$.
\item If $L_t$ is tangent to an ordinary cusp, then $\nabla(b_1)=(\sigma_2\sigma_1)^2$ and the only non-trivial 
relations are given by $a_1=a_3$ and $a_2=a_3 a_2 a_1 a_2^{-1} a_1^{-1}$.
\item If $L_t$ passes transversally through an $m$-tacnode (two smooth branches with intersection number $m$), 
then $\nabla(b_1)=\sigma_1^{2m}$ and the only non-trivial relation is given by $(a_1 a_2)^m=(a_2 a_1)^m$.
\end{itemize}
\end{example}

\begin{rem}
Two remarks about the relations explained in Example~\ref{ex-rels} should be made. First of all, such relations can be 
expressed in such a simple manner when: $t_0$ is close enough to $t$, $b_1$ is a small meridian around $t$, and a suitable 
choice for generators of $\mathbb{F}_d$ is considered (essentially a standard geometric basis). In such cases, 
$\nabla(b_1)$ produces $r$ effective relations, where $r:=d-\#(C\cap L_t)$. 

In the general case, for instance when the base point is not close to $t$, the braid $\nabla(b_1)$ is written as
$\tau^{-1}\sigma\tau$, where:
\begin{itemize}
\item The open braid $\tau^{-1}$ goes from $L_{t_0}$ to a fiber $L_{t_0'}$ close to $L_t$.
\item The braid $\sigma$ is as in Example~\ref{ex-rels} (or a product of these braids involving disjoint subsets of strings).
\end{itemize}
When a (standard) geometric basis $a_1',\dots,a_d'$ in $\pi_1(L_{t_0'}\setminus C)$ is considered, on which $\sigma$ acts, 
only $r$ non-trivial relations are produced. This choice of basis allows one to consider $\tau$ as an element of 
$\mathbb{B}_d$. Using this technique, one can see that $\pi_1(\mathbb{C}^2\setminus C)$ can be described by a system of 
relations of type $a_i'=a_i^{\tau}$.

The relations will involve conjugates of the standard generators of $\mathbb{F}_d$. The number of relations and their 
\emph{type} (equality, commutation, braid relations,\dots) depend only on the braids~$\sigma$, which depend only on 
algebraic properties of $C$ (degree, topological type of singularities), but the involved conjugations of the generators 
depend on the coefficients of $f$. In general, finding the braid $\tau$ explicitly is a very difficult task and unless
the coefficients of $f$ are rational or Gaussian integers, computational methods are far from being efficient.

The braid $\tau$, or equivalently the relationship between the two geometric bases, can be obtained algorithmically
from the real picture, when $f$ has real equations (\emph{real curve}), all (or almost all) of the non-transversal 
vertical lines are real and the branches around the critical points are also real. Such curves are called 
\emph{totally real} curves.
\end{rem}

\section{Fundamental group of hypocycloids and Artin groups}\label{sec_artin}

In this section we compute fundamental groups for some hypocycloids. Hypocycloids are real curves, but not 
totally real curves. However, they are very symmetric and it is by quotienting the plane by these symmetries that one
can obtain a curve that is closer to being totally real.

In these computations a special type of groups, called \emph{Artin groups} will be obtained. Artin groups can be defined 
as follows. Let $\Gamma$ be a finite graph (with no loops and no multiple edges between vertices); let 
$S:=S(\Gamma)$ be the set of vertices and let $E:=E(\Gamma)$ be the set of edges (considered as a subset of 
$\{A\subset S\mid \#A=2\}$). The Artin group $G_\Gamma$ associated with $\Gamma$ is generated by the elements of $S$ 
and has a system of relations given as follows for any $s\neq t$:
\begin{itemize}
\item If $\{s,t\}\in E$ then $s t s=t s t$.
\item If $\{s,t\}\notin E$ then $[s,t]=1$.
\end{itemize}
For example, $\mathbb{B}_d$ is the Artin group associated with the $\mathbb{A}_{d-1}$ graph 
(a connected linear graph with $d-1$ vertices).

\subsection{The Deltoid}
\label{sec-deltoid}
\mbox{}

\begin{figure}[ht]
  \centering
\begin{tikzpicture}[vertice/.style={draw,circle,fill,minimum size=0.4cm,inner sep=0}]

\draw [line width=1pt] plot [parametric,domain=0:2*pi,samples=100]
        ({2*cos(deg(\x))*(1+cos(deg(\x))},{2*sin(deg(\x))*(1-cos(deg(\x))});
\end{tikzpicture}
  \caption{Deltoid}\label{fig_deltoid}
\end{figure}
The deltoid corresponds to $\rho=\frac{1}{3}$, that is, $k=2$, $\ell=1$, and $N=3$. 
In order to obtain an explicit equation for $C_{k,\ell}$, given by parametric equations $(x_{k,\ell}(t),y_{k,\ell}(t))$, 
one needs to compute the resultant of $x_{k,\ell}(t)-x$ and $y_{k,\ell}(t)-y$, with respect to $t$. In this case one obtains
\beq\label{eq-deltoid}
C_{2,1}: 3(x^2+y^2)^2+24 x (x^2+y^2)  +6 (x^2+y^2)-32 x^3-1=0.
\eeq
Let $\pi:\mathbb{C}^2\to\mathbb{C}^2$ be the $2$-fold ramified covering given by $\pi(x,y):=(x,y^2)$.
Let $D_{2,1}:=\pi(C_{2,1})$; since $C_{2,1}$ is symmetric with respect to the involution $\sigma:\mathbb{C}^2\to\mathbb{C}^2$,
$\sigma(x,y):=(x,-y)$ (which generates the automorphism group of $\pi$), one can check that $C_{2,1}=\pi^{-1}(D_{2,1})$, 
where $D_{2,1}$ is given by
\beq\label{eq-deltoid_z2}
D_{2,1}: 3(x^2+y)^2+24 x (x^2+y) +6 (x^2+y)-32 x^3-1=0.
\eeq
In order to compute $\pi_1(\mathbb{C}^2\setminus(D_{2,1}\cup X))$ (where $X$ is the $x$-axis), one has to compute the discriminant
of the equation~\eqref{eq-deltoid_z2} with respect to $y$. Since all its roots are real, that is, $D_{2,1}$ is a totally real
curve, Figure~\ref{fig_deltoid_z2} contains all the topological information needed to compute the group
$\pi_1(\mathbb{C}^2\setminus(D_{2,1}\cup X))$.

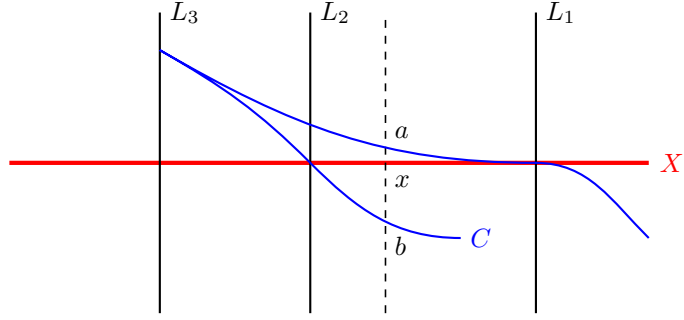
\begin{figure}[ht]
\centering
\begin{tikzpicture}[scale=.5]

\draw[ultra thick,red] (3,8)--(20,8) node[right] {$X$};
\draw[thick,black] (17,4)--(17,12) node[right] {$L_1$};
\draw[semithick,dashed,black] (13,4)--(13,12);
\draw[thick,black] (11,4) -- (11,12) node[right] {$L_2$};
\draw[thick,black] (7,4) -- (7,12)  node[right] {$L_3$};
\draw[thick,blue,smooth] (20,6) to [out=135,in=0] (17,8) to[out=180,in=-30] (7,11)
to[out=-30,in=135] (11,8) to[out=-45,in=180] (15,6)  node[right] {$C$};
\node[right] at (13,8.8) {$a$};
\node[right] at (13,7.5) {$x$};
\node[right] at (13,5.8) {$b$};
\end{tikzpicture}%

\caption{The curve $D_{2,1}\cup X$}\label{fig_deltoid_z2}
\end{figure}
Following~\S\ref{sec-zvk}, the dotted line $L$ represents a generic vertical line (the other three lines 
in Figure~\ref{fig_deltoid_z2} are the non-transversal vertical lines). After fixing a \emph{big enough} real number as the 
base point, one can consider the natural free basis of $\pi_1(L\setminus(D_{2,1}\cup X))$ given by $a,x,b$ 
(positive meridians around the intersections with the curve), such that $(b x a)^{-1}$ is a meridian of the point at infinity.
Moving around the line $L_1$ one obtains the braid $\sigma_1^6$, which produces the relation:
\[
(a x)^3=(x a)^3.
\]
The braid around $L_2$ is $\sigma_2^2$, and the relation is $[b,x]=1$.
In order to compute the relations provided by $L_3$ one can proceed as follows. 
Consider a vertical line $L'$ between $L_2$ and $L_3$ and generators $a',b',x'$ of $\pi_1(L'\setminus(D_{1,2}\cup X))$ 
as done with $L$. In order to see them as elements in  $\pi_1(L\setminus(D_{2,1}\cup X))$ it is necessary to connect 
the base points in the same horizontal line. It is easily seen that $a=a'$, $b=b'$ and $x=x'$. The relation obtained 
around $L_3$ is $a b a=b a b$.

One obtains
\begin{equation*}
\pi_1(\mathbb{C}^2\setminus(D_{2,1}\cup X))=
\langle a,x,b\ \Big\vert\ 
(a x)^3=(x a)^3, [b,x]=1, a b a=b a b
\rangle. 
\end{equation*}
Since $\pi_{|}:\mathbb{C}^2\setminus(C_{2,1}\cup X)\to\mathbb{C}^2\setminus(D_{2,1}\cup X)$
is a double unramified covering, one can check that the group $\pi_1(\mathbb{C}^2\setminus(C_{2,1}\cup X))$
is the subgroup of index~$2$ normally generated by $a,b,x^2$. It is well-known that
$\pi_1(\mathbb{C}^2\setminus C_{2,1})$ can be obtained from $\pi_1(\mathbb{C}^2\setminus(C_{2,1}\cup X))$ 
by adding the relation $x^2=1$.

\begin{rem1}
In fact, the above operations, computing the index~$2$ subgroup and factoring by $x^2$, commute. Therefore, one 
$\pi_1(\mathbb{C}^2\setminus C_{2,1})$ can also be considered as a subgroup of index~$2$ of $G_{2,1}:=\pi_1(\mathbb{C}^2\setminus(D_{2,1}\cup X))/\langle x^2\rangle$.
\end{rem1}

\begin{prop1}
The group $\pi_1(\mathbb{C}^2\setminus C_{2,1})$ is the Artin group of the triangle.
\end{prop1}
\begin{proof}
The group $\pi_1(\mathbb{C}^2\setminus C_{2,1})$ is the kernel of the morphism
$\rho:G_{2,1}\to\langle t\ \mid\ t^2=1\rangle$, given by $\rho(a):=\rho(b):=t$, $\rho(x):=1$.

Using the Reidemeister-Schreier method, $\pi_1(\mathbb{C}^2\setminus C_{2,1})$ is generated by
$a,b,c$ where $c:=x a x$ (note that $b=x b x$, because of the second relation).
The third relation gives $a b a=b a b$ and $c b c=b c b$. The first relation
gives $a c a=c a c$. The presentation of the Artin group for the triangle results directly.
\end{proof}

\begin{rem1}
According to Proposition~\ref{prop-hypo}\ref{part-inf-sing}, the projective closure $\bar{C}_{2,1}$ of $C_{2,1}$ 
in $\mathbb{P}^2$ 
is such that the line at infinity is a bitangent and the contact points are the imaginary cyclic points of order 4. 
The fundamental group $\pi_1(\mathbb{P}^2\setminus \bar{C}_{2,1})$ is a non-abelian finite group of size~$12$ which 
was first computed by Zariski~\cite{zr:29} and it is the braid group of $3$ strings in the $2$-sphere. This is the curve of smallest degree with a non-abelian fundamental group. Its dual 
is a nodal cubic curve (as a real curve with a node with imaginary tangent lines).
\end{rem1}

\subsection{Astroid}
\label{sec-astroid}
\mbox{}

\begin{figure}[ht]
  \centering
\begin{tikzpicture}[scale=.5,vertice/.style={draw,circle,fill,minimum size=0.4cm,inner sep=0}]

\draw [line width=1pt] plot [parametric,domain=0:2*pi,samples=100,rotate=45]
        ({8*cos(deg(\x))^3},{8*sin(deg(\x))^3});
\end{tikzpicture}
\caption{Astroid}
\label{fig_astroid}
\end{figure}

The astroid corresponds to $\rho=\frac{1}{4}$, that is, $k=3$, $\ell=1$, and $N=4$. 
In order to apply a more suitable symmetry, a rotation of the astroid should be performed to obtain a curve as
in Figure~\ref{fig_astroid}. One obtains
\beq\label{eq-astroid}
C_{3,1}: 4(x^2+y^2)^3+15(x^2+y^2)^2+12  (x^2+y^2) -108 x^2 y^2-4=0.
\eeq
As we did for the deltoid in~\S\ref{sec-deltoid}, let us consider $D_{3,1}:=\pi(C_{3,1})$ the quotient of $C_{3,1}$ 
by the symmetry $\sigma$, which has equation
\beq\label{eq-astroid_z2}
D_{3,1}: 4(x^2+y)^3+15(x^2+y)^2+12  (x^2+y) -108 x^2 y-4=0.
\eeq
In order to compute $\pi_1(\mathbb{C}^2\setminus(D_{3,1}\cup X))$ (where $X$ is the horizontal axis), one computes the 
discriminant of the equation~\eqref{eq-deltoid_z2} with respect to $y$. In this case $D_{3,1}$ is not totally real, thus
Figure~\ref{fig_astroid_z2} is not enough to compute the group $\pi_1(\mathbb{C}^2\setminus(D_{3,1}\cup X))$.
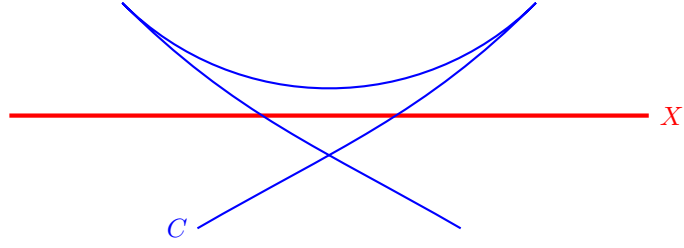
\begin{figure}[ht]
  \centering
\begin{tikzpicture}[scale=.5]

\draw[ultra thick,red] (3,9)--(20,9) node[right] {$X$};
\draw[thick,blue,smooth] (15,6) to [out=150,in=-45] (6,12) to[out=-45,in=-135] (17,12)
 to[out=-135,in=30] (8,6)  node[left] {$C$};
\end{tikzpicture}%

  \caption{The curve $D_{3,1}\cup X$}\label{fig_astroid_z2}
\end{figure}
According to Proposition~\ref{prop-hypo}\ref{part-aff-sing} the astroid has four ordinary (non-real) double points: 
two of them are sent to the \emph{real} double point shown in Figure~\ref{fig_astroid_z2}, the other two are on the 
symmetry axis and are sent to two (non-real) points where $D_{3,1}$ is tangent to $X$, hence $D_{3,1}$ is not a totally
real curve. However, $D_{3,1}$ is symmetric with respect to the $Y:=\{x=0\}$ axis, and hence one can perform the quotient
with respect to this symmetry and obtain the curve $X\cup Y\cup E_{3,1}$, where
\beq\label{eq-astroid_z2z2}
E_{3,1}: 4(x+y)^3+15(x+y)^2+12  (x+y) -108 x y-4=0.
\eeq
The curve is rotated in order to have a better projection.
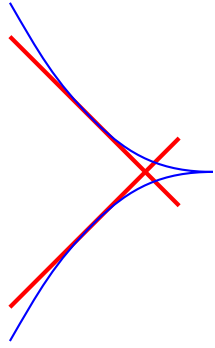
\begin{figure}[ht]
  \centering
\begin{tikzpicture}[scale=.45]

\coordinate (P1) at (15,5);
\coordinate (P2) at (15,-3);
\coordinate (Q1) at (20,0);
\coordinate (Q2) at (20,2);
\draw[ultra thick,red] (P1)--(Q1);
\draw[ultra thick,red] (P2)--(Q2);
\draw[thick,blue,smooth] ($(P1)+(0,1)$) to [out=-60,in=135]  ($.5*(P1)+.5*(Q1)$)  to[out=-45,in=180] (21,1)  to[out=180,in=45] ($.5*(P2)+.5*(Q2)$)
to[out=225,in=60] ($(P2)-(0,1)$) ;
\end{tikzpicture}%

  \caption{The curve $E_{3,1}\cup X\cup Y$}\label{fig_astroid_z2z2}
\end{figure}

According to \S\ref{sec-zvk}, $\pi_1(\mathbb{C}^2\setminus(E_{3,1}\cup X\cup Y))$ is generated by $x,a,b,y$ 
and a system of relations can be given as follows:
$a b a=b a b$, $[a,x]=1$, $[b,y]=1$, $[x,y]=1$, $(a y)^2=(y a)^2$, and $(b x)^2=(x b)^2$.
Defining $G_{3,1}$ as before by adding the relations $x^2=1$, $y^2=1$, one needs to compute the appropriate 
index four subgroup. A straightforward application of the Reidemeister-Schreier method gives the following result.

\begin{prop1}
The group $\pi_1(\mathbb{C}^2\setminus C_{3,1})$ is the Artin group of the square.
\end{prop1}

\begin{rem1}
The curve $\bar{C}_{3,1}$ is a sextic with six cusps and four nodes, i.e., the dual of a nodal quartic.
These curves were studied by O.~Zariski~\cite{zr:36}. The group $\pi_1(\mathbb{P}^2\setminus\bar{C}_{3,1})$
is isomorphic to the braid group of four strings on the sphere.
\end{rem1}

\subsection{Hypocycloid \texorpdfstring{$\rho=\frac{2}{5}$}{rho=2/5}}
\mbox{}

\begin{figure}[ht]
  \centering
\begin{tikzpicture}[scale=.5,vertice/.style={draw,circle,fill,minimum size=0.4cm,inner sep=0}]

\draw [line width=1pt] plot [parametric,domain=0:2*pi,samples=100]
        ({(2*cos(3*deg(\x))+3*cos(2*deg(\x))},{(2*sin(3*deg(\x))-3*sin(2*deg(\x))});
\end{tikzpicture}
  \caption{Hypocycloid $\rho=\frac{2}{5}$}\label{fig_hipo53}
\end{figure}
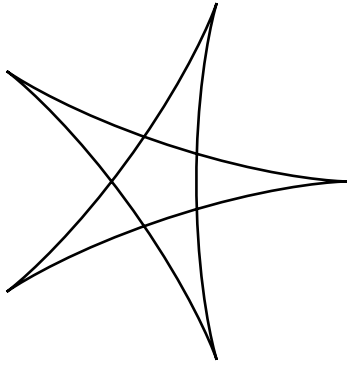
This is a particular case of the hypocycloids for $\rho=\frac{n}{2 n+1}$ where all nodes are real.
The curve $C_{3,2}$ has equation
\beq\label{eq-hipo53}
80(x^2+y^2)^3+165(x^2+y^2)^2-30(x^2+y^2)-216 x ({x}^{4}-10\,{x}^{2}{y}^{2}+5\,{y}^{4})+1
=0.
\eeq
The quotient $D_{3,2}$ of $C_{3,2}$ by the action of $\sigma$ is a totally real curve, and thus the real picture contains 
all the topological information needed to compute the group $\pi_1(\mathbb{C}^2\setminus(D_{2,1}\cup X))$.

\begin{figure}[ht]
  \centering
\begin{tikzpicture}[scale=.4]

\draw[ultra thick,red] (3,9)--(20,9);
\draw[thick,blue,smooth] (18,8)to [out=120,in=0] (16,9) to [out=180,in=-45] (4,14) to [out= -45,in=180] (7,9) to[out=0,in=-135] (13,13) to[out=-135,in=0] (7,7);
 to [out=150,in=-45] (6,12) to[out=-45,in=-135] (17,12)
 to[out=-135,in=30] (8,6)  node[left] {$C$};
\end{tikzpicture}%

  \caption{The curve $D_{3,2}\cup X$}\label{fig_hipo53_z2}
\end{figure}
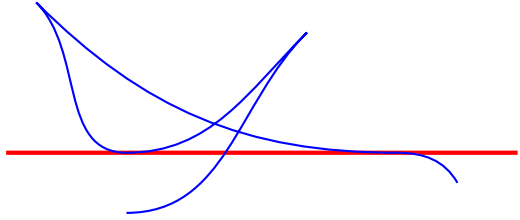

Following \S\ref{sec-zvk}, one needs to select a generic vertical line $L_{t_0}$, which will be chosen to 
sit between the two real double points of $D_{3,2}$. The group $G_{3,2}$, which is the quotient of
$\pi_1(\mathbb{C}^2\setminus(D_{3,2}\cup X))$ by the square of a meridian of $X$,
is generated by $a,b,x,c$ a has the following system of relations $x^2=1$, $[a,b]=1$, $(a x)^2=(x a)^2$, 
$(x a x) b (x a x)=b (x a x) b$, $[c,x]=1$, $[b,c]=1$, $a c a=c a c$, and $(b x)^3=(x b)^3$.

\begin{prop1}
The group $\pi_1(\mathbb{C}^2\setminus C_{3,2})$ is the Artin group of the pentagon.
\end{prop1}

\begin{proof}
After using Reidemeister-Schreier again, one obtains the following set of generators: $a,b,c,u:=x a x, v:=x b x$,
a set of \emph{cuspidal} relations among the adjacent list of generators $u,b,v,a,c$ in a cyclic manner, and 
commutation relations amongst the non-adjacent generators.
\end{proof}

\begin{rem1}
This result was already obtained in~\cite{cer:94,car:94}. The authors computed the correct presentation of the group, but then they stated that this group is isomorphic to 
$\mathbb{B}_5$; this is not true as can be checked either directly or using Artin group theory.
\end{rem1}

\begin{rem1}
With some more computations, it is possible to show that the group for the hypocycloid $C_{4,3}$ (for $\rho=\frac{3}{7}$)
is the Artin group of the polygon of $7$~edges. Since the quotient of the curves $C_{n+1,n}$ is \emph{totally real},
we will compute the general case in a forthcoming paper.
\end{rem1}

\subsection{Hypocycloid \texorpdfstring{$\rho=\frac{3}{8}$}{rho=3/8}}
\mbox{}

\begin{figure}[ht]
  \centering
   \begin{tikzpicture}[scale=.25,vertice/.style={draw,circle,fill,minimum size=0.4cm,inner sep=0}]

\draw [line width=1pt] plot [parametric,domain=0:2*pi,samples=100,rotate=22.5]
        ({(3*cos(5*deg(\x))+5*cos(3*deg(\x))},{(3*sin(5*deg(\x))-5*sin(3*deg(\x))});
\end{tikzpicture}
  \caption{Hypocycloid $\rho=\frac{3}{8}$}\label{fig_hipo85}
\end{figure}
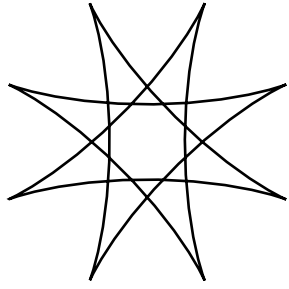
We complete this preliminary study of the topology of the hypocycloid curves
with $C_{5,3}$ corresponding to $\rho=\frac{3}{8}$, i.e., $k=5$, $\ell=3$, and $N=8$. 
As with the astroid in~\ref{sec-astroid}, a rotation will be performed in order to use more suitable symmetries. 
One obtains:
\beq\label{eq-hipo85}
\begin{split}
C_{5,3}:& 11664\, \left( {x}^{2}+{y}^{2} \right) ^{5}+47655\, \left( {x}^{2}+{y}
^{2} \right) ^{4}+40240\, \left( {x}^{2}+{y}^{2} \right) ^{3}-\\
&-17040\,
 \left( {x}^{2}+{y}^{2} \right) ^{2}+1920\,({x}^{2}+{y}^{2})+
1350000\,{x}^{2}{y}^{2} \left( x^{2} -y^2 \right) ^{2
}-64
=0. 
\end{split}
\eeq
Let $D_{5,3}:=\pi(C_{5,3})$. 
\begin{figure}[ht]
  \centering
\begin{tikzpicture}[scale=.4]

\draw[ultra thick,red] (3,9)--(20,9);

\draw[thick,blue,smooth] (12,7)to [out=135,in=-45] (7,15) to [out=-45,in=180] (16,9) to [out= 0,in=225] (20,12) to[out=225,in=0] (11.5,10) to[out=180,in=-45] (3,12) to[out=-45,in=180] (7,9)to[out=0,in=225] (16,15) to[out=225,in=45] (11,7);
 to [out=150,in=-45] (6,12) to[out=-45,in=-135] (17,12)
 to[out=-135,in=30] (8,6)  node[left] {$C$};
\end{tikzpicture}%

  \caption{The curve $D_{5,3}\cup X$}\label{fig_hipo85_z2}
\end{figure}
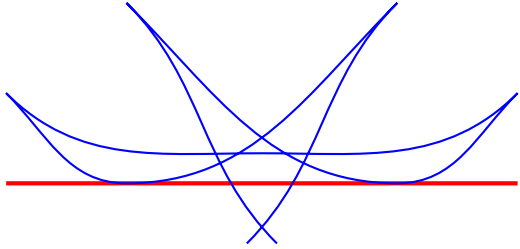

Since $D_{5,3}$ is not totally real, Figure~\ref{fig_hipo85_z2} is not enough to compute 
$\pi_1(\mathbb{C}^2\setminus(D_{5,3}\cup X))$. However, one can use the symmetry along $Y:=\{x=0\}$ to obtain 
the quotient $E_{5,3}$ of $D_{5,3}$. The curve $E_{5,3}\cup X\cup Y$ is not totally real yet. According to 
Proposition~ \ref{prop-hypo}\ref{part-inf-sing}, $C_{5,3}$ has eight imaginary nodes. Since two of them are on $X$, 
they produce another three nodes for $D_{5,3}$, one of them on $Y$ and real. After rotating the axes one obtains
Figure~\ref{fig_hipo85_z2z2}.
\begin{figure}[ht]
  \centering
\begin{tikzpicture}[scale=.5,vertice/.style={draw,circle,fill,minimum size=0.4cm,inner sep=0}]

\draw[ultra thick,red] (-3,-3)--(3,3);
\draw[ultra thick,red] (-3,3)--(3,-3);
\draw[thick,blue,smooth] (-3,-4) to [out=0,in=225]  (-2,-2) to[out=45,in=225] (5,1)
 to[out=225,in=45] (2,2)  to[out=225,in=90] (1,0)  to[out=-90,in=135] (2,-2)
to[out=-45,in=135] (5,-1) to[out=135,in=-45] (-2,2) to [out=135,in=0] (-3,4);
\node[vertice,fill=blue] at (-2,0) {};
\end{tikzpicture}%

  \caption{The curve $E_{5,3}\cup X\cup Y$}\label{fig_hipo85_z2z2}
\end{figure}
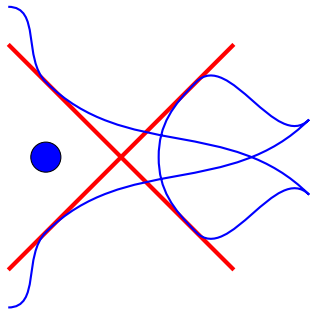
Note that $E_{5,3}$ has a (real) node with imaginary tangent lines and has a symmetry along the line joining 
the node and the origin. 
By this symmetry, $X$ and $Y$ are symmetric to one another.

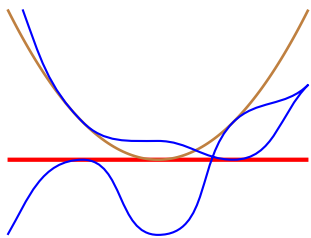
\begin{figure}[ht]
  \centering
\begin{tikzpicture}[scale=1]

\draw[ultra thick,red] (-2,0)--(2,0);

\draw [line width=1pt,brown] plot [parametric,domain=-2:2,samples=100]
        ({\x},{\x*\x/2});

\draw[thick,blue,smooth] (-2,-1) to[out=60,in=180] (-1,0) to[out=0,in=180] (0,-1) to[out=0,in=225] (1,.5)
to[out=45,in=225] (2,1) to[out=225,in=0] (1,0) to[out=180,in=0] (0,.25) to[out=180,in=-45] (-1,.5) to[out=135,in=-70] (-1.8,2);
\end{tikzpicture}%

  \caption{Final curve}\label{fig_hipo85_z2z2z2}
\end{figure}
The resulting quotient is a totally real curve, whose fundamental group can be obtained from the information shown
in Figure~\ref{fig_hipo85_z2z2z2}.

\begin{prop1}
The group $\pi_1(\mathbb{C}^2\setminus C_{3,1})$ is the Artin group of the octagon.
\end{prop1}

\begin{proof}
One has to apply the Reidemeister-Schereier method and use \texttt{GAP}~\cite{GAP4} to obtain the desired presentation.
\end{proof}

\section{Conclusions}

Note that with this method one finds not only the expected groups, but also the expected presentations. This happens also in other computations of the fundamental groups of the complements of special curves, like in~\cite{besmi}.
This seems to suggest that there is a deep geometrical connection between the symmetries and the fundamental group, which should be better understood.

It seems to be possible to generalize the method of \S\ref{sec_artin} for curves $C_{n+1,n}$, since all nodes 
and vertical tangents are real. For the general case, in order to \emph{make visible} all singular points
we may consider the quotient of $\mathbb{C}^2$ by the action of $\mathbb{D}_{2n}$, which is not a smooth surface.


\providecommand{\bysame}{\leavevmode\hbox to3em{\hrulefill}\thinspace}
\providecommand{\MR}{\relax\ifhmode\unskip\space\fi MR }
\providecommand{\MRhref}[2]{%
  \href{http://www.ams.org/mathscinet-getitem?mr=#1}{#2}
}
\providecommand{\href}[2]{#2}

\end{document}